\documentclass[12pt,leqno,twoside]{amsart}
\usepackage[utf8]{inputenc}
\usepackage[T1]{fontenc}
\usepackage[colorlinks=true, pdfstartview=FitV, linkcolor=blue, citecolor=blue, urlcolor=blue]{hyperref}
\usepackage{amstext,amsmath,amscd, bezier,indentfirst,amsthm,amsgen,enumerate, geometry}
\usepackage[all,knot,arc,import,poly]{xy}
\usepackage{amsfonts,color, soul}  %\st{} for overstriking
\usepackage{amssymb}
\usepackage{latexsym}
\usepackage{epsfig}
\usepackage{graphicx}
\usepackage{srcltx}
\usepackage{enumitem}

\newtheorem{theorem}{Theorem}[section]
\newtheorem{corollary}[theorem]{Corollary}
\newtheorem{lemma}[theorem]{Lemma}
\newtheorem{proposition}[theorem]{Proposition}
\newtheorem{pdefinition}[theorem]{Proposition-Definition}
\theoremstyle{definition}
\newtheorem{definition}[theorem]{Definition}
\theoremstyle{remark}
\newtheorem{remark}[theorem]{\sc Remark}
\newtheorem{example}[theorem]{\sc Example}
\renewcommand{\Box}{\square}    %\diamond

\newcommand{\corank}{{\rm{corank}}}

\newcommand{\Sing}{{\mathrm{Sing\hspace{2pt}}}}

\newcommand{\rank}{{\mathrm{rank\hspace{1pt}}}}
\newcommand{\Disc}{{\mathrm{Disc\hspace{2pt}}}}

\newcommand{\im}{{\mathrm{Im\hspace{2pt}}}}

\newcommand{\ity}{{\infty}}
\renewcommand{\d}{{\mathrm{d}}}

\newcommand{\e}{\varepsilon}
\newcommand{\m}{\setminus}
\newcommand{\s}{\subset}

\newcommand{\fin}{\hspace*{\fill}$\Box$\vspace*{2mm}}

\newcommand{\bR}{{\mathbb R}}
\newcommand{\bC}{{\mathbb C}}
\newcommand{\bK}{{\mathbb K}}

\newcommand{\bN}{{\mathbb N}}

\newcommand{\bS}{{\mathbb S}}

\newcommand{\bW}{{\mathbb W}}
\newcommand{\C}{\mathbb{C}}

\begin{document}

\title[Images of analytic map germs and singular fibrations]{Images of analytic map germs, \\ and singular fibrations}
%Map germs: the possibility of a fibration

\author{\sc Cezar Joi\c{t}a}
\address{Institute of Mathematics of the Romanian Academy, P.O. Box 1-764,
 014700 Bucharest, Romania.} % and Laboratoire Europ\' een Associ\'e  CNRS Franco-Roumain Math-Mode}
\email{Cezar.Joita@imar.ro}

\author{Mihai Tib\u{a}r}
\address{Univ. Lille, CNRS, UMR 8524 -- Laboratoire Paul Painlev\'e, F-59000 Lille,
France}  
\email{mihai-marius.tibar@univ-lille.fr}

\dedicatory{In memoriam \c Stefan Papadima}


%\date{July 5, 2019}

\thanks{The authors acknowledge the support of the Labex CEMPI
(ANR-11-LABX-0007-01). The first author acknowledges the
CNCS grant PN-III-P4-ID-PCE-2016-0330.}

\begin{abstract}
For a map germ $G$ with target $(\bC^{p}, 0)$ or  $(\bR^{p}, 0)$ with $p\ge 2$, we address two phenomena which do not occur when $p=1$:  the image of $G$ may  be not well-defined as a set germ, and a local fibration near the origin may not exist.  We show how these two phenomena are related,  and how they can be characterised.
\end{abstract}

\maketitle

\section{Introduction}

We focus here on two phenomena concerning analytic map germs $(\bK^{m}, 0) \to (\bK^{p}, 0)$ with $p\ge 2$, where $\bK = \bR$ or $\bC$, which do not occur when $p=1$:

\medskip

(A).   The image of  a map germ may be not well-defined as a set germ.

% Very simple example show this fact, for instance
%the map germ at the origin $(x,y) \mapsto (xy, x)$.
% some papers contain results which are not completely proved, not well posed or with erroneous proofs.

(B).  A  map germ may not define a local fibration.

\bigskip

In support of  the assertion (A), one of the simplest  examples is the blow-up $F: (\bK^{2}, 0) \to (\bK^{2}, 0)$, 
$F(x,y) = (x, xy)$. The image $F(B_{\e})$  of the ball $B_{\e}$ centred at 0, viewed as a set germ at $0$,  depends heavily on the radius $\e>0$, and therefore the image of the map germ $F$ is not well-defined as a set germ. %  see \cite{JT} for details.

The germ-image condition has been singled out by Mather in \cite[\S 2.5 and \S 9]{Stratif}\footnote{We thank Maria Ruas for drawing our attention to this reference.} as a necessary condition for the construction of a certain Whitney stratification.  However, in the study of map germs initiated 
 by Thom, Milnor, Mather, Arnold etc and continued by many mathematicians until today, the usual setting is  ``map germ $G$ with isolated singularity in  its central fibre $G^{-1}(0)$'', and in this setting
 the germ-image condition turns out to be fulfilled, both over $\bC$ and over $\bR$ -- see below our  Propositions \ref{p:dim} and \ref{l:nice} which treat more general settings.  Our new study concerns the complementary case  ``nonisolated singularities in the central fibre''.

\smallskip

In support of the assertion (B), one may consider the following example taken from \cite{Sa} $F: (\bC^{3}, 0) \to (\bC^{2}, 0)$, $(x,y,z) \mapsto (x^{2}- y^{2}z,  y)$.  Then $F$ is flat, so its image is well-defined as a set germ: $(\im F ,0)  = (\bC^{2},0)$, see e.g. \cite[pag. 214]{BS}.
 Sabbah \cite{Sa} showed that this map germ, with $\Sing F =\{ x=y=0\}$ and $F(\Sing F)=\{ 0\}$, does not have a locally trivial fibration over the set germ $(\bC^{2}\m \{ 0\},0)$.

The following natural questions arise: 
\emph{How can one characterise the phenomena (A) and (B)? Are they related?}

We address here (A) and (B), in this order, for the following reason: the existence of a well-defined image  of a map germ is a necessary condition for the existence of a local fibration, as  already observed in \cite{ART}, but not sufficient, as the above example  shows.  

%It appears that that the converse is true, and this is what we show by Theorem \ref{main-new}, one of our main results.
%Since this is more precise,  before stating it we first need to explain the ingredients.

The recent works \cite{ACT, JT} give partial answers to the question: \emph{under what conditions the image of an analytic map germ is  well-defined as a set germ?} In \cite{JT} we have shown how intricate is the classification even in the cases of holomorphic map germs $(f,g) : (\bC^{n}, 0) \to (\bC^{2}, 0)$ and of real analytic map germs $f\bar g: (\bC^{n}, 0) \to (\bC, 0)$. 

 In the first part of this note (Sections 2 and 3) we introduce the notion of \emph{nice map germs} after \cite{ACT}, abbreviated NMG, we  recall several  key results of \cite{ACT, JT}, and we complete them by the new Propositions \ref{p:dim}, \ref{l:nice} and Theorem \ref{p:sing}. Our Proposition \ref{t:singfbarg} gives more details about the discriminant of the maps $f\bar g$, as studied in  \cite{ParTi, ParTi}.

 %and more generally for complex maps with equidimensional fibres, see Proposition \ref{p:dim}. 
In the second part (Sections 4 and 5) we address  the existence  of a local fibration for map germs
on the basis of the preceding results. The existence problem is actually a long standing one, stemming from Milnor's results \cite{Mi} for holomorphic function germs completed by Hironaka and by L\^e in case of nonisolated singularities,  as well as from Hamm's study \cite{Ha-icis} of the local fibration attached to a complex  isolated complete intersection singularity (abbreviated ICIS, see \cite{Lo, Lo2} for the reference monograph on this topic). 
%The existence of a local fibration is the basis for the study of the topological properties of the  space near the singularity. In %case of holomorphic hypersurfaces  and ICIS, this gave a lot of results in the last 50 years. 
Beyond the ICIS, there are several relatively recent existence results, especially in case of mixed map germs, for instance \cite{dST0, ACT, Ma, ART, Ha, PS, Oka4}.

The
 general notion of ``stratified fibration'' that we work with here will be defined in  \S \ref{s:nmg}  and \S \ref{s:sing}. 
We give the most general condition (the ``tameness'' condition Definition \ref{d:tame}) under which local fibrations exist in  the singular stratified setting.
Our main result is the ``universal criterion for NMG'' Theorem \ref{main-new} telling that ``tame implies NMG'', which  yields a general Fibration Theorem \ref{t:tube} and its corollaries in Section 5.  
 
%One shows in to various degrees of generality (cite ..... and lately \cite{ParTi, ART}) that 
%real and complex  analytic map germs
%$F : (\bK^{n},0) \to (\bK^{p},0)$ may define  local fibrations under certain conditions provided that the images and the %discriminants are well-defined germs in the target.   This quest started fifty years ago with Milnor's study of the Milnor %fibration of holomorphic function germs, 
% and continued ever since into several streams, including the motivic counterpart.

% We prove the existence of a tube fibration whenever the \emph{Milnor condition} is fulfilled.

%We give here a meaning to local fibrations for any map germs.
%%%%%%%%%%%%%
%%%%%%%%%%%%%
\section{Map germs having germ images}\label{s:nice}

\bigskip
%%%%%%%%%%%%%
Let $A, A'\s \bR^{p}$ be  subsets containing the origin and let $(A,0)$ and $(A',0)$ denote their germs at 0. We recall that one has the equality of set germs $(A,0) = (A',0)$ if and only if there exists some open ball $B_{\e}\s \bR^{p}$ centred at 0 and of radius $\e>0$ such that $A\cap B_{\e} = A'\cap B_{\e}$.

\begin{definition}\label{d:nice}
Let $G:(\bR^{m},0) \rightarrow (\bR^{p}, 0)$, $m\ge p >0$, be a continuous map germ.
 We say that the image  $G(K)$  of a set $K\subset \bR^{m}$ containing $0$  is a \emph{well-defined set germ}
 at $0\in \bR^{p}$ if for any small enough open balls $B_{\e},  B_{\e'}$ centred at 0, with  $\e, \e' >0$, we have the equality of germs
 $(G(B_{\e}\cap K),0) =  (G(B_{\e'}\cap K),0)$.
 
  Whenever the images $\im G$ and  $G(\Sing G)$ are well-defined as germs,  one says\footnote{cf \cite{ART}.} that $G$ is a \emph{nice map germ}, abbreviated NMG.
\end{definition}

\begin{remark}
In support of the second part of Definition \ref{d:nice}, let us point out that even if the image $\im G$ is well-defined as a germ, the image of some restriction of $G$ might be not.  This behaviour can be seen in  the following example  $G:\bC^3\to\bC^2$, $G(x,y,z)=(x,z)$ and  $K:= \{(x,y,z) \mid z=xy\}\s \bC^3$, where the image $G(K)$ is not well-defined as a germ. This example is of course based on the first example given in the Introduction.
\end{remark}

% produsul a doua functii mixte  nice nu este neaparat  functie mixta nice:  $(1+z_1)z_2$ si $\bar z_2$ care e anti-holom. 

%%%%%%%%%%%%
%\subsection{Some general results, complex and real settings}\label{ss:complex}\label{ss:real}

\begin{proposition}\label{p:dim}
Let $F:(\C^n, 0)\to(\C^p,0)$,  $n\geq p$, be a holomorphic map germ. If the fibre  $F^{-1}(0)$ has  dimension $n-p$
 then $(\im F, 0) = (\C^{p}, 0)$.  If moreover $\Sing F \cap F^{-1}(0) = \{ 0\}$ then $F$ is a NMG.
 \end{proposition}

\begin{proof}
Let $B\subset \bC^n$ be an open neighbourhood of the origin where the holomorphic map $F$  is defined.
There exists  a closed irreducible analytic subset $Z\subset B$ of dimension $p$, for instance a general complex $p$-plane,  such that $0\in Z$ and that $0$ is  an isolated point of $Z\cap F^{-1}(0)$. Then, by e.g. \cite[Proposition, page 63]{GR}, it follows that there exist
 an open neighbourhood  $U$ of $0$ in $Z$  and   an open neighbourhood $V$ of the origin in $\bC^p$ 
such that $F(U)\subset V$ and the induced map $F_{|U}:U\to V$ is finite. 
By the Open Mapping Theorem  (cf \cite[ page 107]{GR}) this implies 
that $F(U)$ is open, which shows the equality of germs $(\im F, 0) = (\bC^{p}, 0)$.

  %$F$ is said to define a local complete intersection $(F^{-1}(0), 0)$. 
   If moreover $\Sing F \cap F^{-1}(0) = \{ 0\}$ then $(F^{-1}(0), 0)$ is called  an \textit{isolated complete intersection singularity}, abbreviated ICIS, and has been studied in many papers, see Looijenga's book containing the major knowledge until 1984  \cite{Lo, Lo2}.  In the ICIS case the discriminant of $F$, defined as  $\Disc F := F(\Sing F)$, is a complex hypersurface germ, by the purity result \cite[Theorem 4.8]{Lo}. In particular $F$ is a  NMG.  
   \end{proof}

%%%%%%%%
A real counterpart of the above result is the following:
\begin{proposition}\label{l:nice}
Let $G:(\bR^{m},0) \rightarrow (\bR^{p}, 0)$, $m\ge p >0$ be an analytic map germ. 

 If $\Sing G \cap G^{-1}(0) = \{ 0\}$,  then  $G$ is a NMG.
\end{proposition}

%The condition $\Sing G \cap G^{-1}(0) = \{0\}$ is equivalent to  the finite $C^0$-$\mathcal{K}$-determinacy of $G$, see %\cite{Wa,CB} for details. 

\begin{proof} 
Let us assume first that  $\dim G^{-1}(0) >0$. Then one has:
\begin{lemma}\cite{ART}\label{l:im}
 If $\Sing G \cap G^{-1}(0) \subsetneq  G^{-1}(0)$ then  $(\im G,0)= (\bR^{p},0)$.
\end{lemma}
\begin{proof}
Let $q\in  G^{-1}(0) \m \Sing G$, which is nonempty by hypothesis.  Then $G$ is a submersion on some small open
neighbourhood $N_{q}$ of $q$,  thus the restriction $G_{| N_{q}}$ is an open map, and therefore $\im G$ contains some open neighbourhood of the origin of the target.
\end{proof}
In order to show that the image of $\Sing G$ is a well-defined germ, and thus to complete the proof 
that $G$ is a NMG, we use the following lemma:
\begin{lemma}\label{l:haus}
Let $X$ and $Y$  be convenient topological spaces (i.e. Hausdorff, locally compact and 
with countable system of neighbourhoods at each point). Let $f:X\to Y$ be a continuous map, let
$b\in Y$ and let $S$ be a closed subset of $X$. If $S\cap f^{-1}(b)=\{a\}$ then $\im(f_{|S})$ is well-defined as a set germ at $b$.
\end{lemma}
\begin{proof} 
By contradiction, if it is not the case then there exist  two relatively compact 
open neighbourhoods of $a$, $V \supset V'$,  and a sequence of points 
$p_{n}\in Y$, $p_{n} \to b$, $p_{n}\in f(V\cap S)$, $p_{n}\not\in f(V'\cap S)$ for all 
integers $n\gg 1$. Let  then $x_{n}\in V\cap S$ with $f(x_{n}) = p_{n}$. There is a subsequence
$(x_{n_{k}})_{k\in \bN}$ which tends to some point $x$ in the closure $\overline V\cap S$.
We have $f(x) = \lim_{k\to \infty} f(x_{n_k}) = \lim_{k\to \infty}p_{n_{k}}= b$. 
Since $S \cap f^{-1}(b) = a$, the point $x$ must coincide with $a$.  But then one must have 
$x_{n_{k}}\in V'\cap S$ for some $k\gg 1$, which implies that $p_{n_{k}}\in f(V'\cap S)$ which is a contradiction to the assumptions about the sequence $(p_{n})_{n}$.
\end{proof}

Let us now assume $\dim G^{-1}(0) =0$. It  is then sufficient to apply Lemma \ref{l:haus} and we get that both $\im G$ and
$G(\Sing G)$ are well-defined as germs at $0$ in the target.
\end{proof}
%A  non-constant polynomial germ $G:(\bR^{m},0) \rightarrow (\bR^{p}, 0)$, $m\ge p >0$,  
%has a well-defined singular set $\Sing G$. 

%%%%%%%%%%%%%%%%%%%%
%%%%%%%%%%
\section{The case of map germs $(f,g)$ and $f\bar g$}\label{s:(f,g)}

Analytic map germs  $F:(\bC^n, 0)\to (\bC^p, 0)$ for which the dimension of the fibre $F^{-1}(0)$ is greater than $n-p$ are beyond the framework of Proposition \ref{p:dim}. We call them \emph{blow-ups}. 

\subsection{The map germ $(f,g)$}
In \cite{JT} we have given a classification in the case $p=2$, as follows. 
%We complete it here with a precise statement about the image of the singular locus.

 \begin{theorem}\label{t:main1}\cite{JT}
Let $(f,g):(\bC^n,0)\to(\bC^2,0)$ be a non-constant holomorphic map germ. Then:
 \begin{enumerate}
\rm \item \it  If $\dim Z(f)\cap Z(g)=n-2$, then $\im (f,g)$ is a well-defined set germ, more precisely $(\im (f,g), 0) = (\bC^{2}, 0)$.
\rm \item \it  If $\dim Z(f)\cap Z(g)=n-1$ then:
 \begin{enumerate}
 \rm \item \it   in case $Z(g)\subset Z(f)$ or  $Z(f)\subset Z(g)$, the image $\im (f,g)$ is  a well-defined set germ  if and only if $\im (f,g)$ is an irreducible plane curve $(C,0)$.  
\rm \item \it   in case   $Z(f)\not\subset Z(g)$ and $Z(g)\not\subset Z(f)$, the image 
$\im (f,g)$ is a well-defined set germ if and only if $(\im (f,g), 0)=(\bC^{2}, 0)$.
 \end{enumerate}
 \end{enumerate}
 \fin
\end{theorem}

 The following result together with Theorem \ref{t:main1} provide 
 the solution to the NMG problem in case of  map germs $(f,g)$.
 
\begin{theorem}\label{p:sing}
 For any map germ $F= (f,g) :(\bC^n, 0)\to (\bC^2, 0)$, $n\ge 2$,  
 the image of the singular locus   $(f,g)(\Sing (f,g))$ is  well-defined as a complex analytic set germ.
\end{theorem}

 \begin{remark}
 The above result tells in particular  that $F(\Sing F)$ is a set germ even if  $\im F$ may be not well-defined as a set germ at 0.
 More precisely, if it is non-empty, then the set $(f,g)(\Sing (f,g))$ is either the origin $0\in \bC^{2}$, or  a complex analytic (not necessarily irreducible) curve.
 
However this  result does not hold for map germs with $p\geq 3$. \\
Example:
$F(x,y,z)=(xy,y,z^2)$, with $\Sing F=\{y=0\}\cup \{z=0\}$. The image $F(\{z=0\})$ is not a set germ at $0$, and therefore $F(\Sing F)$ is not an analytic set germ at 0.
\end{remark}

\subsection{Proof of Theorem \ref{p:sing}}
  Let $\Sing(f,g) = \cup_{j=1}^{k}S_{j}$ be the decomposition into positive dimensional irreducible components of the analytic germ $\Sing(f,g)$.  
We fix some component $S:= S_{j}$ and prove that the image $(f,g)(S)$  is well-defined as an analytic set germ at 0.

If $(f,g)_{|S}\equiv 0$, then our claim is trivially true, so let us assume $(f,g)_{|S}\not\equiv 0$. 
Then  $\dim_x(f_{|S},g_{|S})^{-1}(f(x),g(x))$ is either $\dim S -2$, or $\dim S -1$, for any $x\in S$. Indeed, the dimension of the fibre cannot be smaller than $\dim S -2$; if it is greater than $\dim S -1$,  i.e. equal to $\dim S$, then the fibre must be equal to the irreducible component $S$, which amounts to $(f,g)_{|S}\equiv 0$, and this contradicts the hypothesis.

As for the ranks, we have the inequality  $\rank_{x}(f_{|S},g_{|S})\leq  \rank_{x}(f,g)$ for any $x\in S\m \Sing(S)$, and we have $\rank_{x}(f,g)\leq 1$ by the definition of the singular locus. 
Since $\rank(f_{|S},g_{|S})\equiv 0$ implies $(f,g)_{|S}\equiv 0$,  
we only have to deal with the case   $\rank(f_{|S},g_{|S})\not\equiv 0$. 

Then 
$S_{0}:=\{x\in S\m \Sing(S)  \mid \rank_x(f_{|S},g_{|S})>0\}$  is an open, connected and dense analytic subset of $S$, and actually
 $\rank_x(f_{|S},g_{|S})=1$, $\forall x\in S_{0}$. By the rank theorem, 
we deduce that
$\dim_x(f_{|S},g_{|S})^{-1}(f(x),g(x))=\dim S -1$ for all $x\in S_{0}$.  The next result on the semi-continuity of the dimension of fibres is useful in order to figure out what happens at points in $S\m S_{0}$.

\begin{lemma}\cite[page 66]{Na}
Let $F:X\to Y$ be a holomorphic map between complex spaces. Then every $a\in X$ has 
a neighbourhood $U\subset X$ such that 
\[ \dim_x F^{-1}(F(x))\leq \dim_a F^{-1}(F(a))\]
  for any $x\in U$.
  \fin
\end{lemma}

Indeed, since $S_{0}$ is dense in $S$,  the above lemma yields in our case that  for any $x\in S$ one has at least the inequality $\dim_x(f_{|S},g_{|S})^{-1}(f(x),g(x))\geq \dim S -1$.  However, this inequality cannot be strict (i.e. not even at points in $S\m S_{0}$) since the converse inequality  $\dim_x(f_{|S},g_{|S})^{-1}(f(x),g(x))\leq \dim S -1$ necessarily holds, as we have shown in the first part of the above  proof. 

And now, since  $\dim_x(f_{|S},g_{|S})^{-1}(f(x),g(x))=\dim S -1$
for all $x\in S$,  
our claim follows from the next result applied to the point $0\in S$, result which is also called  ``Remmert's Rank Theorem''  by \L ojasiewicz in his book \cite[Theorem 1, pag 295]{Loj2}:

\begin{lemma}\label{p:nara}\cite[Prop. 3, Ch. VII]{Na}
Let  $f:X\to Y$ be a  holomorphic map between complex spaces such that $X$ is  purely dimensional. If $\dim_x f^{-1}(f(x)) =r$ is independent of $x\in X$, then any point $a\in X$
has a fundamental system of neighbourhoods $U$ such that the image $f(U)$ is analytic at $f(a)$, of dimension $\dim X - r$.
\fin
\end{lemma}
This shows that the image $(f,g)(S\cap B)$ is an analytic set at $0\in \bC^{2}$, for any small enough ball $B$ centred at $0\in \bC^{n}$. Therefore this image is well-defined  as an analytic set germ at $0$, namely  it is either the point $0$, or an irreducible plane curve germ (as it cannot be the whole target space, by Sard theorem). This ends our proof of Theorem \ref{p:sing}.

%%%%%%%%%%%%%%%%%%

%%%%%%%%%%%%
\bigskip
\subsection{The map germ $f\bar g$}
To holomorphic function germs $f, g: (\bC^n,0)\to (\bC, 0)$  one associates the particular real map germ $f\bar g: (\bC^n,0)\to (\bC, 0)$, in particular its singular locus is well-defined as the singular set of this real map germs, see e.g. \cite{CT}. In  \cite{ParTi} one studied  $f\bar g$ in relation to  the holomorphic map germ $(f,g): (\bC^n,0)\to (\bC^{2}, 0)$. 
The following result is a corrected version of the statement in \cite{ParTi}.
%\footnote{The full corrected proof is also contained in the updated preprint version of \cite{ParTi}.}

\begin{proposition}\label{t:singfbarg}\rm \cite[Theorem 2.3 and Lemma 2.5]{ParTi} \it
Let $f,g : (\bC^n,0)\to (\bC, 0)$, $n>1$,  be some non-constant holomorphic function germs. 
  Then $f\bar g (\Sing f\bar g)$ is a well-defined semi-analytic set germ of dimension $\le 1$. 
  Its dimension is precisely 1 if and only if the germ $(f,g)(\Sing (f,g))$ contains at least one irreducible component tangent to some line different from the axes. 
\end{proposition}

\begin{proof}[Sketch of the proof]
The map $f\bar g$ decomposes as $\bC^n \stackrel{(f,g)}{\rightarrow} \bC^2 \stackrel{u\bar v}{\rightarrow} \bC$.  
It is not difficult to see (cf \cite{ParTi})  that $\Sing f\bar g \subset \Sing (f,g) \cup \{fg = 0\}$, thus it remains to study 
the singular loci of the restrictions $(u\bar v)_{|C}$, for all irreducible complex curve germ components $C\subset (f,g)(\Sing (f,g))$, and we have shown in Theorem \ref {p:sing} that $(f,g)(\Sing (f,g))$ is a well-defined complex analytic set germ of dimension $\le 1$.

 One shows that the restriction $(u\bar v)_{|C \m \{ 0\}}$ is a submersion if and only if  the curve germ $C$ is tangent to one of the coordinate axes without coinciding with it.
To do that one considers a Puiseux expansion $u=t^p$, $v= a_1 t^q + \mbox{h.o.t.}$ of  $C$, with $a_1\not= 0$. 
The equality  which defines the singular locus of the function $u(t)\bar v(t)$ is:
\begin{equation}\label{eq:singbar}
  v \overline{\d u} = \lambda u \overline{\d v},
\end{equation}
 where $| \lambda | = 1$, and note that $\lambda$ depends on $t$ (see e.g. \cite{Oka0, CT}). 
Taking the modulus on both sides gives $|v \d u|= |u\d v|$, and thus $| p a_1 t ^{p+q-1} + \mbox{h.o.t.}|
= | q a_1 t ^{p+q-1} + \mbox{h.o.t.}|$. One gets the equality
$p=q$, which implies that,  if $C$ is tangent to one of the coordinate axes without coinciding with it,  then the restriction $(u\bar v)_{|C \m \{ 0\}}$ is a submersion.

Reciprocally, if $p=q$ we first look at the case of  a one-term expansion $v= a_1 t^p$. Then all points $t$ are solutions of  equation \eqref{eq:singbar}, and  consequently the corresponding critical value set contributes with a real  half-line in $f\bar g (\Sing f\bar g)$, as $C$ is a line different from the coordinate axes by the assumption $a_1\not= 0$. 

If  the Puiseux expansion $v= a_1 t^p + a_2 t^{p+j} + \mbox{h.o.t.}$ has at least 2 terms, where $j\ge 1$, then the equivalent equation  $| v \d u | =| u \d v |$ reduces, after dividing out by $p| a_{1}| |t|^{2p-1}$, to an equation of the form
$| 1 + bt^{j}+ \mbox{h.o.t.}| =| 1 + ct^{j}+ \mbox{h.o.t.}|$, with $b,c\not= 0$ and $b\not= c$.  This has as solutions a positive  number of semi-analytic arcs $\gamma(s)$  parametrised by $s\in [0, \e[$, for some small enough $\e>0$. The reason is that, for any fixed modulus $|t|$, one has $2j$ points of intersection of the circle $S^{1}$ of radius 1 with the closed curve which is a slightly perturbed small circle centred at $(1,0)$,  covered $j$-times due the variation of the argument of the variable $t$ within one period. The image  $(u\bar v)(\gamma)$ of such an arc is included in $f\bar g (\Sing f\bar g)$.  Since by hypothesis neither $u$ nor $v$ are constant along this arc,  it follows that this image is  not reduced to the point $0$, hence it must be a non-trivial continuous real arc. This arc is in fact semi-analytic, by \L ojasiewicz' result saying that the image by an analytic map of a real semi-analytic arc is a semi-analytic arc (see e.g. \cite[\S 8]{Su}).
Moreover, by the above proof, if $\Sing f\bar g \not= \emptyset$, then $f\bar g (\Sing f\bar g)$  is either the origin $0$, or the union of all the images $(u\bar v)(\gamma)$ over all arcs $\gamma$, and over all complex curve components $C\subset ((f,g)(\Sing (f,g)), 0)$. In particular $f\bar g (\Sing f\bar g)$ is well-defined as a semi-analytic set germ.
\end{proof}

The following result is the fusion of \cite[Theorem 1.1]{JT} and the above Proposition \ref{t:singfbarg},
and provides a
 complete answer to the NMG problem for $f\bar g$, 
 based on the preceding NMG  classification of the maps $(f,g)$:

\begin{theorem}\label{p:fbarg-nice} 
Let $f \bar g: (\bC^{n}, 0)\to (\bC^{2}, 0)$, $n>1$, for some non-constant holomorphic germs $f$ and $g$. 
 \begin{enumerate}
\rm \item \it  If $f\not= ug$ for any invertible function germ $u$, then $f\bar g$ is a NMG, and $(\im f\bar g, 0) = (\bC, 0)$.  
\rm \item \it  If $f= ug$ for some invertible function germ $u$, then  $f\bar g$ is a NMG if and only if
$\im (f,g)$ is a complex curve germ.
 \end{enumerate}
 \fin
\end{theorem}

\begin{example}\label{ex:fbarg}
$(f,g):\bC^2\to \C^{2}$, $(f,g)(z,w)=(z(1+w),z)$ is an example for Theorem \ref{p:fbarg-nice}(b) where $\im f\bar g$ is not a well-defined set germ. One has  $(f\bar g) (\Sing f\bar g)=\{0\}$.
%The Milnor set is: $u(1+u)+v^2=x^2+y^2$.
%See also Example \ref{ex:hansen}.
\end{example}

%%%%%%%%%%%%
%%%%%%%%%%%%
% \section{The existence of fibrations for map germs}
 \section{Milnor set and nice map germs }\label{s:nmg}
 %%%%%%%%%%%%%%%%
 
 Let us now consider the general case of  a non-constant  analytic map germ $G:(\bR^{m},0) \rightarrow (\bR^{p}, 0)$, $m\ge p >1$.
 
 We recall that whenever the images $\im G$ and  $G(\Sing G)$ are well-defined as germs, then we say that $G$ is a \emph{nice map germ}, abbreviated NMG (Definition \ref{d:nice}).
 
Our main new result of this section is a convenient condition which implies NMG in full generality. 
It is very close to  problem (B) quoted in the Introduction, it almost coincides with 
the criterion for the existence of the stratified Milnor fibration. This fact sheds new light over the recent results \cite{ART, ART2} which use the NMG condition as a preliminary hypothesis for the existence of this fibration.   Briefly speaking we show here that we can remove the NMG condition from the statements about the local singular fibration. 
%We therefore update some statements

\begin{pdefinition}\label{d:discr}
 If $G$ is a NMG, then the \emph{discriminant}  $\Disc G := G(\Sing G)$ is a well-defined closed subanalytic set germ.
 \end{pdefinition}
 
\begin{proof} %We first show that $(\im G, 0)$ is a \emph{closed} subanalytic set germ. 
Since $\im G$ is well-defined as a set germ,  for some small enough balls centred at the origin  $B_1\subset \overline B\subset B_2$ we have $(G(B_1),0)= (G(B_2),0)=(G(\overline B),0)$. Thus  $\im G$ is a closed subanalytic set, since it is the well-defined germ of the analytic image of a compact set.

 Next we show that the boundary $(\partial \im G, 0)$ is contained in  $(G(\Sing G), 0)$. Let us remark that $\partial \im G$ is either empty, or  a closed subanalytic  proper subset of $\bR^{p}$ which is well-defined  as a set germ at $0\in \bR^{p}$.     Since $\im G$ is closed, we have the inclusion $\partial \im G \subset \im G$. Our claim then follows by the fact that the image by $G$ of the complement of $\Sing G$ is an open set germ.

 Finally we show that $(G(\Sing G), 0)$ is a closed subanalytic set germ as follows. For some small enough balls centred at the origin  $B_1\subset \overline B\subset B_2$, the  NMG property implies the equality of the set germs
$(G(\Sing G\cap B_1),0)= (G(\Sing G\cap B_2),0)=(G(\Sing G\cap \overline B),0)$, and therefore 
$G(\Sing G)$ is the well-defined germ of the analytic image of a compact set.
\end{proof}

\begin{remark}\label{r:discr}
In \cite{ACT} one defined the discriminant as the union $\overline{G(\Sing G)} \cup \partial \im G$. The above proof  shows  that,  a posteriori, this reduces to the ``classical'' definition of the discriminant $\Disc G := G(\Sing G)$, as in Definition \ref{d:discr}, provided of course that $G$ is a NMG. 

As $\Disc G$ is closed,  its complement is well-defined as a set germ at the origin,
and it is the disjoint union of finitely many open connected subanalytic set germs.  
\end{remark}

%%%%%%%%%%%%%%%%%%%%%%
\subsection{The  Milnor set $M(G)$} \   \\
Let  $G:(\bR^{m},0) \rightarrow (\bR^{p}, 0)$ be a non-constant  analytic map germ, $m \ge p \ge1$.
Let $U \subset \bR^m$ be a connected manifold,  and let	
 $M(G_{|U}):=\left\lbrace x \in U \mid \rho_{|U} \not\pitchfork_x G_{|U} \right\rbrace $
	be the set of \textit{$\rho$-nonregular points} of $G_{|U}$, or \emph{the Milnor set of $G_{|U}$}, where
	  $\rho := \| \cdot \|$ denotes here the Euclidean distance function, and $\rho_{|U}$ is its restriction to $U$.  We tacitly consider $M(G_{|U})$ as a set germ at 0.
	  
	  By definition $M(G_{|U})$ coincides with the singular set  $\Sing (\rho, G)_{|U}$ defined in its turn as 
the set of points $x\in U$ such that  either  $\rank_{x}(G_{|U})$ is zero, or $\rank_{x}(G_{|U})$ is less than the maximal rank denoted $r_{U}$ and $r_{U}>0$,    or $\rank_{x}( \rho_{|U}, G_{|U}) = \rank_{x}(G_{|U}) = r_{U}>0$.  

Equivalently, the point $y\in U$ is not in $M(G_{|U})$ if and only if $\rank(G_{|U})$ is constant $>0$ in some neighbourhood of  $x$  and  $\rank_{x}( \rho_{|U}, G_{|U}) > \rank_{x}(G_{|U})$.

In the following we will actually consider the germ at 0 of $M(G_{|U})$. 
  It turns out from the definition that it is real analytic.

\medskip	  

  In case of map germs $G$ of rank $p$ outside the origin, it is clear from the definition that one has the property $M(G_{|U})\cap G^{-1}(0) \subset \{0\}$, which is called ``$\rho$-regularity'' condition, and  which Milnor \cite{Mi} exploited  to show  the existence of a well-defined tube fibration.  Later, in many papers this $\rho$-regularity was spelled out as a necessary condition for the existence of a well-defined fibration, either locally or in its version at infinity  \cite{Be, Ti1,  Ti2, Ti3, dST0, dST1, Ma,  ACT, ACT-inf} etc. In case of $G$ with non-isolated singularities, the $\rho$-regularity has to be properly defined 
 and it is no more automatic; we shall discuss this matter below, abutting to Definition \ref{d:tame}. For the use of the $\rho$-regularity in the stratified setting we also refer to \cite{Be} and to the recent papers \cite{ART, ART2}.

% \begin{definition}\label{d:stratif}
	%Let  $G:(\bR^{m},0) \rightarrow (\bR^{p}, 0)$ be a non-constant  analytic map germ, $m\ge p >1$.
\medskip
	 	There exists a germ of a finite semi-analytic Whitney (a)-stratification\footnote{As shown originally by \L ojasiewicz \cite{Loj} for the much stronger Whitney (b)-condition, see e.g. \cite{Ka} and  its references to previous proofs in the literature. See also \cite{NTT} for a nice geometric proof and its relation to \cite{Ka}.}  $\bW_G = \{ W_\alpha\}_\alpha$  of the source of $(\bR^{m},0)$ 
such that the restriction of $G$ to every stratum\footnote{By definition, strata are connected manifolds.}  is a submersion to its image.  In particular every stratum is a nonsingular, open and connected semi-analytic set, and each restriction $G_{|W_\alpha}$ has constant rank.  
%We consider the roughest (i.e. the less refined) such stratification, 
%and we call it \emph{the Whitney stratification  of  $G$ at 0}. 
By definition the subset $\Sing G$ is a union of strata, and each connected component of its complement is a stratum.

%\end{definition}	 
\begin{definition}\label{d:Mstr} \ 
Let  $G:(\bR^{m},0) \rightarrow (\bR^{p}, 0)$, with $m\geq p >1$, be a non-constant  analytic map germ, and let $\bW$ be a finite semi-analytic Whitney (a)-stratification of $G$ at 0.

Let $W_\alpha \in \bW$  denote the germ of some stratum, and let 
$M(G_{|W_\alpha})$  %:=\{x \in W_\alpha \mid \rho_{|W_\alpha} \not\pitchfork_x G_{|W_\alpha} \}$
be the Milnor set of $G_{|W_\alpha}$. 
One calls $M(G):=\sqcup_{\alpha} M(G_{|W_\alpha})$
	the set of \textit{stratwise $\rho$-nonregular points} of $G$ with respect to the stratification $\bW$.
\end{definition}

Note  that $M(G)$  is  closed because $\bW$ is a Whitney (a)-stratification  (see e.g. \cite[Def. 2.2.1]{Ti3}).
%The non-transversality condition reads, equivalently: $\rank( \rho_{|W_\alpha}, G_{|W_\alpha}) = \rank(G_{|W_\alpha})$.
 For instance it follows directly from the definition that any stratum included in $G^{-1}(0)$ is in $M(G)$, and that
 we have the implication:  
  \[ \rank(G_{|W_\alpha})=\dim W_\alpha >0  \implies W_\alpha \subset M(G).\]

Let us also note that if $x\not\in M(G)$, $x\in W_{\alpha}$, then the fibre  $G^{-1}(G(x))\cap W_{\alpha}$
has positive dimension, it is non-singular at $x$, and transverse  at $x$ to the sphere $\{\rho=\rho(x)\}$.

\begin{definition}\label{d:tame}
Let $G:(\bR^{m},0) \rightarrow (\bR^{p}, 0)$, with $m \geq p >1$, be a non-constant  analytic map germ. We say that \emph{$G$ is tame} if it satisfies the following condition, viewed as an inclusion of set germs:
\begin{equation}\label{eq:main2}
\overline{M(G) \m  G^{-1}(0)} \cap G^{-1}(0) \subset     \{ 0\}.
\end{equation}
 \end{definition}

This appeared as a condition for the existence of a local Milnor fibration in a conical neighbourhood of $G^{-1}(0)$
 \cite[Proof of Theorem 1.3]{ACT}. In the particular setting $\Sing G \subset G^{-1}(0)$,  condition \eqref{eq:main2} was needed for the existence of a local Milnor tube fibration in \cite{Ma}, \cite[Proposition 5.3]{dST0} and \cite[Theorem 2.1]{ACT}. Here we are beyond these settings, we work in a highly singular situation,  and we are interested in the existence of the image as a set germ as a preliminary condition in the  problem of the existence of singular local fibrations.

It follows from the definition that if $G$ is tame then the closure of the strata of dimensions $\le p$ intersect  $G^{-1}(0)$
only at $\{ 0\}$. A particular case of tame maps is: $G$ with $G^{-1}(0) = \{0\}$ as a set germ at 0. 

% The  above condition \eqref{eq:rho} is the one adapted to our setting of a ``large'' singular locus; it allows us to prove the %following:

\subsection{Universal criterion}

%Let $G:(\bR^m, 0) \to (\bR^p,0)$, $m> p >0$, be a non-constant analytic map. 

%Let $U$ be a bounded neighborhood of $0$ in $\bR^m$, with real-analytic boundary, such that $G$ 
%is defined on a neighborhood of $\overline U$.

\begin{theorem}\label{main-new}
Let  $G:(\bR^{m},0) \rightarrow (\bR^{p}, 0)$, with $m \geq p >1$, be a non-constant  analytic map germ.
If $G$ is tame then:
\begin{enumerate}
 \rm \item \it $G$ is a NMG at the origin.
 \rm \item \it the image  $G(W_\alpha)$ is a well-defined set germ at the origin, for any stratum $W_\alpha \in \bW$.
 %\rm \item \it  $\Disc G = G(\Sing G)$.
\end{enumerate}
\end{theorem}
\begin{remark}
That $\im G$ is well-defined as a set germ has been already proved in the following  particular cases:

\smallskip
\noindent
(i). $\Sing G \subset  G^{-1}(0)$ and  $G$ satisfies condition \eqref{eq:main2}.  \\
%In this case the connected components of $\bR^m \m G^{-1}(0)$ are strata of $\bW_G$.\\
It was shown  by Massey \cite[Cor. 4.7]{Ma} that $(\im G,0) = (\bR^p,0)$.  In \emph{loc.cit.}, the two conditions listed at (i) were called   ``Milnor conditions (a) and (b)'', respectively. %; they are a particular case of the condition  $M(G)_{e}= \emptyset$. 

\smallskip
\noindent
(ii).  $\Sing G \cap G^{-1}(0) \not= G^{-1}(0)$.  \\
This is  our  Lemma \ref{l:im}. It also follows that $(\im G,0) = (\bR^p,0)$ in this case too. Notice that condition  \eqref{eq:main2} is not required here.

\medskip 
In what concerns the image of $G$ only,  the following more subtle case remains to be proved:

\noindent
(iii). $\Sing G$ includes $G^{-1}(0)$ strictly,  and  $G$ is tame.

\medskip
Moreover,  in the above non-trivial cases (ii) and (iii), it also remains to prove that the image of $\Sing G$  by $G$ is a well-defined set germ.
\end{remark}
\begin{proof}[Proof of Theorem \ref{main-new}]

Let $X\subset \bR^m$ be a closed subset containing $0$, and let $h:(X, 0)\to (\bR^p, 0)$ be a continuous map germ.  We define :
\[ N_h:=\{ x\in X \mid \rho_{|X}(x)=\min\{\rho_{|X}(y) \mid y\in h^{-1}(h(x))\}\}. \]
Since $X$ is closed, $N_h$ intersects each non-empty fibre of $h$. We then prove the following criterion for a continuous map germ to have a well-defined image germ.

\begin{lemma}\label{l:N}
The image of $h$ is a well-defined germ at the origin if and only if
$0$ is an isolated point of $\overline N_h\cap h^{-1}(0)$.
\end{lemma}

\begin{proof}
If  $0$ is an isolated point of $\overline N_h\cap h^{-1}(0)$, then let $B_0$ be a ball centred at the 
origin such that $\overline B_0\cap \overline N_h\cap h^{-1}(0)=\{0\}$. We claim that for any two open balls
centred at the origin, $B$ and $B_1$, with $B_1\subset B\subset B_0$ we have the equality of set germs $(h(B),0)=(h(B_1),0)$.
Indeed, suppose that this is not the case. Then there exists a sequence of points
$\{y_n\}_{n\in \bN}$, $y_n\in h(B)\setminus h(B_1)$ and $y_n\to 0$. We may choose $x_n\in B \cap N_h\cap h^{-1}(y_n)\neq\emptyset$. Since $y_n\not\in h(B_1)$ we have
$x_n\not\in B_1$. The bounded sequence $\{x_n\}$  has some convergent subsequence, let then $\tilde x$ be the limit of this  subsequence. We then have $\tilde x\in \overline B \cap \overline N_h\cap h^{-1}(0)$, which by our hypothesis implies
$\tilde x = 0$. But our construction yields $\tilde x\not\in B_1$, which is contradictory. 

Conversely, suppose that the image of $h$ is a well-defined germ at the origin. We choose a ball $B$
such that, for any ball centred at the origin $B_1\subset B$, we have the equality of set  germs  $(h(B),0)=(h(B_1),0)$. 
We claim that 
$B\cap \overline N_h\cap h^{-1}(0)=\{0\}$. Suppose that this is not true and let 
$\tilde x\in B\cap \overline N_h\cap h^{-1}(0)$ with  $\tilde x\neq 0$. 
We choose a ball $B_1$, centred at the origin, such that 
$\tilde x\not \in \overline B_1$, and a sequence $\{x_n\}_{n\in \bN}\subset N_h$ such that $x_n\to\tilde x$.
Since $\tilde x\in B\setminus \overline B_1$ one may assume that $x_n\in B\setminus \overline B_1$ 
for all $n$. For $y_n:=h(x_n)$,  one then has $y_n\in h(B)$ and $y_n\to 0$. From 
$x_n\in N_h$ and $x_n\not \in B_1$ we deduce that $h^{-1}(y_n)\cap B_1=\emptyset$ and therefore 
$y_n\not \in h(B_1)$. We thus have $y_n\in h(B)\setminus h(B_1)$ for all $n$, with $y_n\to 0$,  and this contradicts
the assumed equality of germs $(h(B),0)=(h(B_1),0)$.
\end{proof}

In order to finish the proof of Theorem \ref{main-new}, we
apply Lemma \ref{l:N}. 
Let us observe that $N_{G}$ is a subset of $M(G)$ since each minimum point $x\in N_{G}$ is a critical point of the distance function restricted to the stratum to which $x$ belongs. Then the condition of Lemma \ref{l:N} is insured by the tameness condition \eqref{eq:main2}.

\noindent
Part $(a)$ is an application of Lemma \ref{l:N} for $h=G$ and $X:= \bR^{m}$,  and then for the restriction $h:=G_{|\Sing G}$.  

\noindent
For part (b) we first apply Lemma \ref{l:N} to the restriction of $G$ to $X:= \overline{W_\alpha}$, for all strata $W_\alpha \in \bW$. Then we extract the information for the image of the open strata recursively, starting from the lowest dimension. It follows that the image $G(W_\alpha)$ of each stratum is well-defined as a set germ at 0.
\end{proof}

\begin{remark}
If $X$ and $h$ are analytic then $N_h$ is subanalytic, since it can be defined as follows: if $A:=\{(x,y)\in
X\times X \mid h(x)=h(y), \ \rho(x)>\rho(y)\}$ and $p:X\times X\to X$, $p(x,y)=x$, then $N_h=X\setminus p(A)$. Here $X$ may be assumed compact since we deal with germs, and hence $h$ is proper.

One may also remark that the condition $\overline N_G\cap G^{-1}(0) = \{ 0\}$ does not imply that $G$ is tame whereas the reciprocal is always true. As an example, 
Sabbah's example considered in the Introduction does not have a tube fibration \cite{Sa}, and therefore is not tame according to our Theorem \ref{t:tube}.
\end{remark}

%%%%%%%%%%%
\subsection{Examples with  $(\im G, 0)\not= (\bR^p,0)$}

We give below two examples for the situation $(\im G, 0)\not= (\bR^p,0)$, one of which is tame and the other is not.
\begin{example}\label{ex:lessimage}\cite[Example 6.7]{ART}
$G : (\bR^{3}, 0) \to (\bR^{2}, 0)$, 
 $G(x,y,z)= (xy, z^{2})$ has $\Sing G = \{ z=0\} \cup \{x=y=0\}$ and $G^{-1}(0) =  \{x=z=0\}\cup \{y=z=0\}$ thus $G^{-1}(0)\subsetneq \Sing G$.  The  regular stratification $\bW$ can be described as follows: 
 
 - the strata of dimension 2 are the connected components of the plane $\{ z=0\}$ minus the two axes $\{x=z=0\}\cup \{y=z=0\}$, and their images are the two components of the dotted line $\bR \times \{ 0\} \m \{(0,0)\}$.
 
 - the strata of dimension 1 are the connected components of the union of the 3 coordinate axes minus the origin; the image of the dotted $z$-axis is $ \{ 0\}\times \bR_+$ whereas the other two axes are sent to $\{(0,0)\}$.

 We get $M(G) = \{ x=\pm y\}\cup \{ z=0\} \subset \bR^3$, and thus condition \eqref{eq:main2} is not fulfilled because of the surface component $\{ z=0\}$, so $G$ is not tame. However $G$ is nice, with
  $\im G = \bR \times \bR_{\ge 0}\not= \bR^{2}$,    $\partial \im G = \bR\times \{0\} \subset G(\Sing G) = \{0\} \times \bR_{\ge 0} \cup \bR\times \{0\}$.
\end{example}

\begin{example}\label{ex:hansen} 
Writing Example \ref{ex:fbarg} in real coordinates we obtain
$F:\bR^4\to\bR^2$, 
\[F(x,y,u,v)=((x^2+y^2)(1+u),(x^2+y^2)v). \]
We have 
$\Sing F =  \{x=y=0\} = F^{-1}(0)$, and
\[M(F)=\{u(1+u)+v^2=\frac{x^2+y^2}2\}\cup\{x=y=0\}.\]
If we restrict $F$ to the hyperplane $\{u=0\}$, we obtain Hansen's example  (see \cite{Ha, ART}) 
$G : (\bR^{3}, 0) \to (\bR^{2}, 0)$, 
\[ G(x,y,v)= (x^{2}+ y^{2}, v(x^{2}+ y^{2}))\]
 for which 
$\Sing G =  \{x=y=0\} = G^{-1}(0)$, and $M(G) = \bR^{3}$.  

In both examples the map germs are not tame, and  one can easily see that
the images are not well-defined as  set germs.
\end{example}

%%%%%%%%%%%
%%%%%%%

\section{Tame maps and the existence of the singular Milnor tube fibration}\label{s:sing}
%%%%%%%%%%%%%%%%%%%%%

%\section{Existence of Milnor-Hamm tube fibration}
%%%%%%%%%%

One of the main motivation for the study of the image of map germs is the study of the local fibrations.  We are concerned with the case of positive dimensional discriminant. One has introduced in \cite{ACT}  the notion of \emph{singular 
Milnor fibrations} and has proved a fundamental existence criterion.  Here we give a sharper  one: we show that tame map germs are endowed with a local singular fibration.  

\smallskip

Let  $G:(\bR^{m},0) \rightarrow (\bR^{p}, 0)$ be a non-constant  analytic map germ, $m> p >1$, and let 
$\bW$ be the Whitney stratification  of $G$ at 0. Assuming from now on that $G$ is tame, our Theorem \ref{main-new}   tells that the 
images of all strata of $\bW$ are well-defined as set germs at 0. This implies that, by using the classical stratification theory, see e.g. \cite{Hi, DSW, LSW},  one may stratify the target too.  More precisely,
%associate to $G$ 
  there exists a germ of a finite subanalytic stratification $\bS$   of  the target such that  $\Disc G$ is a union of strata,  and that $G$ is  a stratified submersion relative to the couple of stratifications $(\bW, \bS)$,
 meaning that the image by $G$ of a stratum  $W_\alpha \in \bW$ is  a single stratum of $S_{\beta} \in \bS$, 
	and that the restriction $G_{|}:W_\alpha \to S_{\beta}$ is a submersion.
	One calls  $(\bW, \bS)$ a \emph{regular stratification of the map germ $G$.}
	
	Such a stratification has been introduced in \cite[Def. 6.1]{ART}\footnote{Without assuming the condition \eqref{eq:main2}.} and the map germs $G$ for which this exists are called \emph{S-nice}.  
	Then, as a consequence of   Theorem \ref{main-new}, we get:
\begin{corollary}\label{c:tamenice}
Let  $G:(\bR^{m},0) \rightarrow (\bR^{p}, 0)$, with $m\ge p >1$, be a non-constant  analytic map germ.
If $G$ is tame then $G$  is S-nice.	
\fin
\end{corollary}

\begin{remark}\label{r:tame}
S-nice map germs may exist even without assuming the condition \eqref{eq:main2}. 
By Theorem \ref{t:main1} and Theorem \ref{p:sing}, the map germs of type $(f,g)$ such that $\gcd(f,g) =1$ are S-nice.
Indeed, we get that $\Disc (f,g) = (f,g)(\Sing(f,g))$ consists of only 1-dimensional complex strata and 0 as the single point-stratum.
Similarly, by Proposition \ref{t:singfbarg} and Theorem \ref{p:fbarg-nice}, the map germs $f\bar g$ such that $\gcd(f,g) =1$ are S-nice map germs.
\end{remark}

\begin{definition}\label{d:tube1}\cite[Definition 6.3]{ART}
Let $G:(\bR^{m},0) \rightarrow (\bR^{p}, 0)$, $m\ge p>1$,  be a non-constant S-nice analytic map germ.
We say that $G$ has  a \emph{singular Milnor tube fibration} relative to some regular stratification $(\bW, \bS)$, which is well-defined as a germ at the origin by our assumption,  if for any small enough $\e > 0$ there exists  $0<\eta \ll \e$ such that the restriction:
\begin{equation}\label{eq:tube1}
G_| :  B^{m}_{\e} \cap G^{-1}( B^{p}_\eta \m \{ 0\} ) \to  B^{p}_\eta \m \{ 0\} 
\end{equation}
is a stratified locally trivial fibration which is independent, up to stratified homeomorphisms, of the choices of $\e$ and $\eta$. 
\end{definition}

Here is what means more precisely ``independent, up to stratified homeomorphisms, of the choices of $\e$ and $\eta$'':  when replacing $\e$ by some $\e'<\e$ and $\eta$ by some small enough $\eta'<\eta$, then the  fibration  \eqref{eq:tube1} and the analogous fibration for $\e'$ and $\eta'$
have the same stratified image in the smaller disk $B^{p}_{\eta'} \m \{ 0\}$, and the fibrations are stratified diffeomorphic over this disk. This property is based on the fact that the image of $G$ is well-defined as a stratified set germ, which amounts to our assumption of ``S-nice''.

By \emph{stratified locally trivial fibration} we mean that for any stratum $S_{\beta}$, the restriction $G_{| G^{-1}(S_{\beta})}$ is a locally trivial stratwise fibration.
The non-empty fibres are those over some connected stratum $S_{\beta} \subset \im G$ of $\bS$. Each such fibre is a singular stratified set, namely it is the union of its intersections with all strata $W_{\alpha}\subset G^{-1} (S_{\beta})$.

%%%%%%%%%%%%%%

\medskip

The following existence theorem gives a fundamental shortcut to the existence  theorem \cite[Theorem 6.5]{ART} by eliminating the condition  ``S-nice'' from the hypotheses. %Its proof is essentially the same, we give some more details here.

\begin{theorem} \label{t:tube}
	Let $G:(\bR^m, 0) \to (\bR^p,0)$, $m \ge p>1$,  be a non-constant analytic map germ. 	
	If $G$ is tame, then $G$ has a singular Milnor tube fibration \eqref{eq:tube1}.
\end{theorem}

%%%%%%%%%%%%%%

\begin{proof} 
Let us fix a  regular stratification $(\bW, \bS)$. By definition, the restriction of $G$ to any stratum  $W_{\alpha}\in \bW$ is  nonsingular and of constant rank.  

Let us first consider the strata $W_{\alpha}$ such that the fibres of  the restriction $G_{|W_{\alpha}}$ are of dimension $>0$, equivalently $\corank(G_{|W_{\alpha}})\ge 1$.
Condition \eqref{eq:main2} implies the existence of  $\e_0>0$ such that, for any  $0<\e <\e_0$, there exists $\eta$, $0<\eta \ll \e$, such that, for any stratum $S_{\beta}\in \bS$, $S_{\beta}\subset G(W_{\alpha})$,  the restriction map
\begin{equation}\label{eq:restrG}
G_{|}:  W_{\alpha}\cap \overline{B^m_\e} \cap G^{-1}(B^{p}_\eta \m \{0\}) \to  S_{\beta} \cap  B^{p}_{\eta} \m \{0\}
\end{equation}
is a submersion on a manifold with boundary.  Indeed, 
since the sphere $S^{m-1}_\e$ is transversal to all the finitely many strata of the Whitney stratification $\bW$ at $0\in \bR^{n}$,
it follows that  the intersection $S^{m-1}_\e \cap \bW$ is a Whitney stratification $\bW_{S, \e}$ of $S^{m-1}_\e$.  Condition \eqref{eq:main2} tells that the restrictions of the map $G$ to strata of $\bW$ is stratified-transversal to the strata of the stratification $\bW_{S, \e}$, for any $0<\e < \e_{0}$. 

As for the strata	$W_{\alpha}$ such that the fibres of  the restriction $G_{|W_{\alpha}}$ are of dimension $0$, we have seen that they belong to $M(G)$, by definition. Thus, since these strata do not intersect the sphere boundary of the source in \eqref{eq:restrG},  the restrictions of $G$ to such strata are proper. 

It then follows that the map $G$ is a proper stratified submersion, thus it is  a stratified fibration by Thom-Mather Isotopy Theorem.  Condition \eqref{eq:main2} also implies that this fibration is independent of $\e$ and $\eta$ up to stratified homeomorphisms. 
\end{proof}

%The case $\dim G^{-1}(0)=0$, which was excepted from Theorem \ref{t:tube}, is equivalent to the equality
%$G^{-1}(0)=\{0\}$ of  set germs at 0. Even without the tameness of $G$, we then have:
\begin{remark}\label{r:dim0}
  If $G:(\bR^m, 0) \to (\bR^p,0)$, $m\ge p>1$, is a non-constant analytic map germ with $G^{-1}(0)=0$, then $G$ is a NMG and has a singular Milnor tube fibration \eqref{eq:tube1}.  Notice that we do not require here the tameness of $G$.
  
Let us prove our claim. Indeed, 
  the NMG property follows trivially from Proposition \ref{l:nice}. Moreover $G$ is a proper map. Indeed, if $G$ is non-proper then, for any small enough $\e>0$, there exists a sequence of points $(x_{n})_{n\in \bN}\subset B_{\e}$ such that $\lim_{n\to \ity} \| x_{n}\| >0$ and $\lim_{n\to \ity} G(x_{n}) =0$. But this is impossible because $G^{-1}(0)=0$. Since our $G$ is a proper stratified submersion, one concludes like in the above proof of Theorem \ref{t:tube}.
%\end{proof}

\end{remark}

\subsection{Milnor-Hamm fibrations over the complement of $\Disc G$.} \ \\
In \cite{ART} one has studied the existence of a fibration over the complement of the discriminant, whenever
the map is a NMG. Such fibrations have been called \emph{Milnor-Hamm fibrations} in \emph{loc.cit}.
We may now remark that under condition \eqref{eq:main2} our proof shows the existence of stratified fibrations over each stratum of the target. In particular  the Milnor-Hamm fibration outside the discriminant $\Disc G$  is well-defined since by construction $\Disc G$ is a union of strata.  One may then derive from Theorem \ref{t:tube} the following Milnor-Hamm fibration result:

\begin{corollary}\label{c:tube-hamm}
Under the hypotheses of Theorem \ref{t:tube} (or of the Remark \ref{r:dim0}, respectively),  the map germ $G$ has a Milnor-Hamm fibration over $B^{p}_{\eta} \m \Disc G$, with nonsingular Milnor fibres over each connected component. 
\fin
\end{corollary}

Let us remark that one may have the Milnor-Hamm fibration  without the existence of the singular Milnor tube fibration. This is the case in Example \ref{ex:lessimage}.

%%%%%%%%%%%%%%%%%%%%%%%%%%%
\subsection{Relation with the Thom regularity, after \cite{ParTi, ART}}
We explain here the relation between condition \eqref{eq:main2} and the Thom regularity, a well-known notion which is  involved in the existence of local fibrations in the classical settings of ICIS or  $\Disc G \subset G^{-1}(0)$, see for instance \cite{Le, Sa, Ti1,  Ti3,  Ma, DRT,  dST1, CGS, Oka4, ParTi}, and more recently beyond these settings \cite{MS, ART, ART2}.
We recall\footnote{After \cite{GLPW, JM}.} that given some stratification of a neighbourhood of $0 \in \bR^{m}$,  a stratum $A$ is \textit{Thom regular} over a stratum $B\subset \bar A\m A$ at $x\in B$ relative to $G$ (or, equivalently,  that the pair $(A,B)$ satisfies the Thom (a$_G$)-regularity condition at $x$),  if the following condition holds: for any $\{x_n\}_{n\in \bN}\subset A$ such that  $x_n\to x$, if  $T_{x_n}(G_{|A})$ converges  to a limit $H$
in the appropriate Grassmann bundle,  then  $T_{x}(G_{|B})\subset H$. 

\begin{definition}\label{d:thom-general} \cite{ART} 
Let $G:(\bR^m, 0) \to (\bR^p,0)$, $m \geq p>1$,  be a non-constant analytic map germ. One says that 
$G$ is \textit{Thom regular at $G^{-1}(0)$} if there
exists a regular stratification $(\bW, \bS)$ such that $G^{-1}(0)$ is a union of strata of $\bW$, that $0$ is a point-stratum in $\bS$, and that the Thom (a$_{G}$)-regularity condition is satisfied at any stratum of $G^{-1}(0)\m \{ 0\}$.
\end{definition}

In the above definition,  one only needs that Thom regularity holds at the fibre over 0. This condition is enough to insure condition \eqref{eq:main2} and thus we derive the following statement from Theorem \ref{t:tube}:

\begin{corollary}\label{c:tube-gen}
	Let $G:(\bR^m, 0) \to (\bR^p,0)$, $m \geq p>1$, be a non-constant  analytic map germ. If $G$ is Thom regular at $G^{-1}(0)$, and if $\dim G^{-1}(0) >0$,  then $G$ is a tame NMG,  and has a singular Milnor tube fibration \eqref{eq:tube1}. 
	\fin
\end{corollary}

\subsection{Thom regularity of map germs $f\bar g$}

 %Some comment on testing the Thom regularity in particular cases (isolated singularity): algebraic test by Gaffney, with %Newton polyhedra methods for special nondegenerate germs, see Oka.

One has shown in \cite[Proof of Theorem 3.1]{ParTi} that if the map germ $(f,g)$ is Thom regular then the map germ $f\bar{g}$  is Thom regular too.  For instance this is the case when $(f,g)$ defines an ICIS.
Combining this  with Corollary \ref{c:tube-gen} yields the following improved formulation of 
 \cite[Theorem 4.3]{ART}:
 
\begin{corollary}\label{c3} \label{t:mainfbarg}
 Let $f, g : (\mathbb{C}^n,0) \to (\mathbb{C},0)$, $n > 1$, be  non-constant holomorphic functions.
 %such that the map germ $(f,g)$ is nice. Then one has the inclusion $\NT_{f\bar{g}} \subset \NT_{(f,g)}$.

If  the map $(f,g)$ is Thom regular, then $f\bar{g}$ is a NMG, it is Thom regular and has a local singular Milnor fibration.
In particular this is the case if $(f,g)$ defines an ICIS.
\fin
\end{corollary}

Corollary \ref{c3}  already provides classes of map germs  $f\bar{g}$ with positive dimensional discriminant  which are S-nice,  Thom regular at $(f\bar{g})^{-1}(0)$,
and thus have singular Milnor tube fibration.  Here is an example with this property:

\begin{example}\cite{ART}
Let $f,g :\mathbb{C}^2 \to \mathbb{C}$ given by $f(x,y) =xy+x^{2}$ and $g(x,y)= y^{2}$. One has $(f,g)^{-1}(0,0) = \{(0,0)\}$ and $\Sing(f,g) = \{y=0\}\cup \{y=-2x\}$, thus $(f,g)$ is obviously  Thom regular. However $\Disc (f,g) = \{(x^2,0)\,|\, x\in\mathbb{C} \}\cup \{(-x^2,4x^2)\,|\, x\in\mathbb{C} \}$ and therefore $f\bar g$ has  non-isolated critical value, namely $\Disc f\bar g$ is the real negative semi-axis. It then follows from Corollary \ref{c:tube-gen}  that $f\bar{g}$ is Thom regular, hence it  has a Milnor-Hamm fibration, and also a singular Milnor tube fibration.
\end{example}
\medskip

Let us finally remark that 
the singular Milnor tube fibration  may exist without the Thom regularity, see \cite{Ti4},  \cite[\S 5.3]{ACT}, \cite[Example 6.11]{ART}, \cite{Ri}.

%\begin{example} Consider $G(x,y,z)=(xy,z^2).$ By hand calculations, we get that it satisfies the singular Milnor tube fibration, but it is not %full. The problem is that it has 20 (twenty) Whitney strata in the source and 4 (four) Whitney strata on the target. One has three different %types of fibers: empty, hyperbolas and cross-cap $(x.y=0)$. 
%\end{example}

%%%%%%%%%%%%%%%
%%%%%%%%%%%%%

%%%%%%%%%%%%%%%%%%%%

%%%%%%%%%%%%%%%%%%%%%%%%%%%

\end{document}